\documentclass[11pt,a4paper,reqno]{amsart}
\usepackage{graphicx} %

\usepackage{times} %
\usepackage{amsmath} %
\usepackage{amssymb}  %
\usepackage{amsthm}
\usepackage{latexsym}
\usepackage{amsfonts,bbm}
\usepackage{xcolor}
\usepackage{mathtools}
\usepackage{enumerate}
\usepackage{cite}
\usepackage{tikz}
\usepackage{graphicx,caption}
\usepackage{todonotes}
\usepackage{float}
\usepackage{nicefrac}
\usepackage[foot]{amsaddr}
\usepackage{color}
\usepackage{graphicx}

\usepackage{tikz}
\usetikzlibrary{shapes.geometric}
\usetikzlibrary{arrows.meta,arrows}

\usepackage{tabularx}
\usepackage{comment}
\usepackage[hidelinks, colorlinks=false]{hyperref}
\usepackage{cleveref}

\usepackage[normalem]{ulem}

\usepackage{algorithm}

\allowdisplaybreaks

\DeclareMathOperator*{\argmin}{arg\,min}

\newcommand{\ceil}[1]{\lceil #1 \rceil}

\usepackage{amsmath,amssymb,amsfonts}
\usepackage{latexsym}
\usepackage{mathrsfs}
\usepackage{bbm}
\usepackage{xcolor}
\usepackage{xcolor}
\usepackage{tikz}
\usepackage{url}
\usepackage{stmaryrd}
\usepackage{dsfont}
\usepackage{placeins}
\usepackage{lipsum}
\usepackage{amsfonts}
\usepackage{graphicx}
\usepackage{epstopdf}
\usepackage{amsopn}
\ifpdf
  \DeclareGraphicsExtensions{.eps,.pdf,.png,.jpg}
\else
  \DeclareGraphicsExtensions{.eps}
\fi

\newcommand{\R}{\ensuremath{\mathbb R}}  
\newcommand{\N}{\ensuremath{\mathbb N}}  

\newtheorem{remark}{Remark}
\newtheorem{definition}{Definition}
\newtheorem{theorem}{Theorem}
\newtheorem{proposition}{Proposition}
\newtheorem{lemma}{Lemma}
\newtheorem{assumption}{Assumption}

\renewcommand{\k}{\mathsf{k}}
\newcommand{\F}{F^\varepsilon}
\newcommand{\g}{g^\varepsilon}
\newcommand{\G}{G^\varepsilon}
\newcommand{\U}{\mathcal{U}^\varepsilon}

\newcommand{\V}{V^\varepsilon}
\newcommand{\J}{J^\varepsilon}
\newcommand{\B}{B^\varepsilon}

\allowdisplaybreaks

\title{Kernel EDMD for data-driven nonlinear Koopman MPC with stability guarantees}
\author{Lea Bold$^{1}$} \address{$^{1}$Optimization-based Control Group, Institute of Mathematics, Technische Universität Ilmenau, Germany.\\ Mail: \textsc{\{lea.bold, karl.worthmann\}@tu-ilmenau.de}}
\author{Manuel Schaller$^{2}$}\address{$^{2}$Faculty of Mathematics, Chemnitz University of Technology\\ Mail: \textsc{manuel.schaller@math.tu-chemnitz.de}}
\author{Irene Schimperna$^{3}$}\address{$^{3}$Civil Engineering and Architecture Department, University of Pavia, Italy \\ Mail: \textsc{irene.schimperna01@universitadipavia.it}}
\author{Karl Worthmann$^{1}$}
\thanks{L.~Bold is grateful for the support of the German Research Foundation (DFG; project number 545246093).}

\begin{document}

\begin{abstract}                % Abstract of 50--100 words
    Extended dynamic mode decomposition (EDMD) is a popular data-driven method to predict the action of the Koopman operator, i.e., the evolution of an observable function along the flow of a dynamical system.
    In this paper, we leverage a recently-introduced kernel EDMD method for control systems for data-driven model predictive control. Building upon pointwise error bounds proportional in the state, we rigorously show practical asymptotic stability of the origin w.r.t.\ the MPC closed loop without stabilizing terminal conditions. The key novelty is that we avoid restrictive invariance conditions. Last, we verify our findings by numerical simulations.
\end{abstract}

\maketitle

\section{Introduction}

\noindent The Koopman operator and its data-driven approximation via extended dynamic mode decomposition (EDMD) are popular and powerful tools for analysis, prediction, and control of dynamical systems, see \cite{mauroy2020introduction}. The Koopman operator is a linear but infinite-dimensional operator to predict the behavior of an associated (nonlinear) dynamical system along observable functions. For autonomous systems, EDMD leverages finitely many samples of observable functions to obtain a linear surrogate model of the system in a lifted space, see, e.g., \cite{williams2015data}. 
Various approaches extend this idea to control systems. In EDMD with control (EDMDc; \cite{proctor2016dynamic}), a surrogate control system that is linear in (lifted) state and control is used. 
However, as shown in \cite{iacob:toth:schoukens:2022}, this method is often insufficient to capture direct state-control couplings. Bilinear models as in \cite{peitz:otto:rowley:2020} offer a reliable alternative, which is underpinned by finite-data error bounds introduced in \cite{nuske2023finite}.
A key challenge is the choice of the observable functions.
Here, data-driven approaches, such as neural networks (\cite{yeung2019learning}) or kernel EDMD provide a remedy. Building upon the approximation-theoretic foundation via reproducing kernel Hilbert spaces (RKHSs), finite-data error bounds 
were recently deduced in~\cite{kohne2024infty}. 
Importantly in view of safe data-driven controller design, these bounds are pointwise and deterministic, see~\cite{bold2024kernel} for an extension to control systems and \textit{proportional error bounds} guaranteeing that the error vanishes at a desired set point

When employing data-driven models in predictive control, error bounds are crucial to show, e.g., stability properties. 
For predictive control of linear systems, a powerful data-driven surrogate is given by the fundamental lemma in \cite{WillRapi05} providing an exact representation of the input-output behavior, see also the recent survey~\cite{FaulOu23}. 
For nonlinear systems, EDMD surrogates have already been successfully employed as prediction models in model predictive control (MPC), e.g., in~\cite{korda2018linear, peitz:otto:rowley:2020}. 
However in EDMD-based approaches, exactness of the data-driven model is a restrictive assumption due to the finitely-many samples and observables. 
However, building upon error bounds on EDMD, it is possible to exploit the nominal robustness of MPC, see \cite{grimm2007nominally}, to show that the control law designed using the data-driven surrogate stabilizes the original system, see, e.g., \cite[Section~4.2]{bold2024kernel}. 
Under invariance assumptions on the dictionary as proposed in~\cite{GoswPale21}, this was proven for EDMD-based surrogates in \cite{WortStra24,bold2024data} with and without terminal conditions.

Our contribution is to leverage kernel EDMD and the corresponding error analysis derived in~\cite{bold2024kernel} to prove stability without imposing invariance assumptions. 
Invoking cost controllability of the original 
system, we rigorously show practical asymptotic stability of the origin w.r.t.\ the MPC closed loop. To this end, we leverage the proportional error dervied in~\cite{bold2024kernel} to prove that cost controllability is preserved for the kEDMD approximant, where we have control on the approximation accuracy.
\\

\noindent \textbf{Notation}. The comparison functions used in the following are $\alpha \in C(\R_{\ge0}, \R_{\ge0})$ is said to be of class~$\mathcal{K}$, if it is strictly increasing with $\alpha(0) = 0$ and of class $\mathcal{K}_\infty$ if it is, in addition, unbounded. 
Furthermore, a function $\delta \in C(\R_{\ge0}, \R_{\ge0})$ is said to be of class $\mathcal{L}$ if it is strictly decreasing with $\lim_{t\rightarrow\infty} \delta(t) = 0$. Lastly, 
$\beta \in C(\R^2_{\ge0}, \R_{\ge0})$ is of class $\mathcal{KL}$ if $\beta(\cdot, t) \in \mathcal{K}$ and $\beta(r, \cdot) \in \mathcal{L}$.
The notation $\|\cdot\|$ is used for the Euclidean norm on $\R^n$ and its induced matrix norm on $\R^{n\times n}$. For a number~$d\in\N$, the abbreviation $[1:d] := \mathbb{Z} \cap [1,d]$ is used. By $C_b(\Omega)$, we denote the space of bounded continuous functions on a set $\Omega\subset\R^n$. 
We require the following continuity notion.

\section{Kernel EDMD and Control}\label{sec:kEDMD}

\noindent Let $\Omega \subset \R^n$ be an open and bounded set with Lipschitz-boundary with $0 \in \operatorname{int}(\Omega)$. For a map~$F:\Omega \rightarrow \R^n$, we consider the discrete-time dynamical system given by
\begin{align}\label{eq:dynamics:F}\tag{DS}
    x^+ = F(x).
\end{align}
In the following, we assume that the set~$\Omega$ is forward invariant w.r.t.\ the dynamics~$F$ of~\eqref{eq:dynamics:F}, i.e., 
$F(x) \in \Omega$ for all $x \in \Omega$, to streamline the presentation.\footnote{We refer to~\cite{kohne2024infty} for the necessary (technical) modifications if this assumption does not hold.}
The associated (infinite-dimensional) linear and bounded Koopman operator~$\mathcal{K}:C_b(F(\Omega)) \rightarrow C_b(\Omega)$ is defined by the identity 
\begin{align*}
    (\mathcal{K} \psi)(\hat{x}) = \psi(F(\hat{x})) \qquad \forall \ \hat{x} \in \Omega, \ \psi \in C_b(F(\Omega)).
\end{align*}

\noindent \textbf{Recap of reproducing kernel Hilbert spaces}.
Let $\k: \R^n \times \R^n \rightarrow \R$ be a symmetric and strictly-positive-definite kernel function, i.e., for every set~$\mathcal{Z} = \{z_1, \dots, z_p\} \subset \R^n$ of pairwise distinct elements, the \textit{kernel matrix} $K_\mathcal{Z} = (\k(z_i, z_j))_{i, j = 1}^p$ is positive definite. 
For $z \in \R^n$, the \textit{canonical features}~$\Phi_z$ of $\k$ are defined by $\Phi_z(x) = \k(z, x)$, $x \in \R^n$. 
By completion, the kernel~$\k$ induces an Hilbert space~$\mathbb{H}$ of functions with inner product $\langle \cdot,\cdot \rangle_{\mathbb{H}}$, see \cite{Wend04}. Importantly, the elements $f \in \mathbb{H}$ fulfill
\begin{align*}
    f(x) = \langle f, \Phi_x \rangle \qquad\forall\,x \in \mathbb{R}^n,
\end{align*}
i.e., the reproducing property, such that $\mathbb{H}$ is termed \emph{reproducing kernel Hilbert space} (RKHS).

\begin{remark}[Wendland kernels]\label{rem:wendland}
    In our work, we use piecewise polynomial and compactly supported kernel functions based on the Wendland radial basis functions (RBF)~$\Phi_{n, k}: \R^n \rightarrow \R$ with smoothness degree~$k \in \N$. 
    In our numerical simulations in Section~\ref{sec:simulations}, $k = 1$ is used, which corresponds for $n \leq 3$ to $\phi_{n, 1} \in \mathcal{C}^{2}([0,\infty),\mathbb{R})$ defined by
    $\phi_{n, 1}(r) = \frac{1}{20}(1-r)^4 (4r+1)$ for $r < 1$ (and $0$ otherwise).
    The induced kernel is given by
    \begin{align*}
        \k(x,y) = \phi_{n,1}(\|x-y\|) \qquad \text{for } x, y \in \R^n,
    \end{align*}
    see \cite[Table~9.1]{Wend04}. The RKHS induced by the Wendland kernels coincides with fractional Sobolev spaces with equivalent norms, a key property, which also holds for Matérn kernels; see \cite{FassYe11}. 
\end{remark}

\noindent \textbf{Kernel EDMD for autonomous systems}. 
For a set of $d \in \N$ pairwise distinct points
\begin{align}\label{eq:Xcal}
    \mathcal{X} = \{x_1, \dots, x_d\} \subset \Omega,
\end{align} 
we define the $d$-dimensional subspace of features $V_\mathcal{X} = \operatorname{span}\{\Phi_{x_1}, \dots, \Phi_{x_d}\}\subset \mathbb{H}$.
Further, $P_\mathcal{X}$ denotes the orthogonal projection onto $V_\mathcal{X}$, i.e., for $f \in \mathbb{H}$, $P_\mathcal{X} f$ solves the regression problem 
    $\min_{g \in V_\mathcal{X}} \|g - f\|^2_\mathbb{H}$.
By the reproducing property, $P_\mathcal{X} f$ interpolates~$f$ at $\mathcal{X}$.
As introduced in \cite{kohne2024infty}, we consider the matrix approximant $\widehat{K}$ of $P_\mathcal{X}\mathcal{K}|_{V_{\mathcal{X}}}$, which is given by
\begin{align}\label{eq:Koopman}
    \widehat{K} = K_\mathcal{X}^{-1} K_{F(\mathcal{X})} K_\mathcal{X}^{-1} \in \R^{d \times d}, 
\end{align}
where $K_{F(\mathcal{X})} = (\k(x_i, F(x_j)))_{i, j = 1}^d$. This corresponds to kernel Extended Dynamic Mode Decomposition (kEDMD).  
For an observable function $\psi \in \mathbb{H}\subset C_b(F(\Omega))$, the surrogate dynamics are given by
\begin{align*}
    \psi(F(x)) \approx \psi^+(x) := \sum_{i = 1}^d (\widehat{K} \psi_\mathcal{X})_i \Phi_{x_i}(x), 
\end{align*}
where $\psi_\mathcal{X} = (\psi(x_1), \dots, \psi(x_d))^\top$.
A bound on the full approximation error was established in~\cite[Theorem~5.2]{kohne2024infty} and is provided in the following theorem. Hence, the approximation accuracy of kEDMD can be described using the \textit{fill distance}~$h_\mathcal{X}$ of the set $\mathcal{X}$ in $\Omega$, which is defined by
\begin{align*}
    h_\mathcal{X} = \sup_{x \in \Omega} \min_{x_i \in \mathcal{X}} \|x - x_i \|.
\end{align*}
\begin{theorem}
    \label{thm:koehne}
    Let $\mathbb{H}$ be the RKHS on~$\Omega$ generated by the Wendland kernels with smoothness degree $k \in \N$.
    Further, let $\mathcal{X} :=\{ x_i \mid i \in [1:d] \} \subset \Omega$ be a set of finitely-many pairwise-distinct data points and the right-hand side ${F \in C^p(\overline{\Omega},\R^n)}$ of system~\eqref{eq:dynamics:F}, with $p = \lceil \frac{n + 1}{2} + k\rceil$. 
    Then there are constants $C,h_0 > 0$ such that the bound on the full approximation error 
    \begin{align}\label{eq:error_bound}
        \| \mathcal{K} - \widehat{K}\|_{\mathbb{H} \rightarrow C_b({\Omega},\R^n)} \le Ch_\mathcal{X}^{k+1 \mathbin{/} 2}
    \end{align}
    holds for all fill distances~$h_{\mathcal{X}}$ satisfying $h_\mathcal{X} \leq h_0$.
\end{theorem}

We note that for a large number of data points in $\mathcal{X}$, the computation of the approximation of the Koopman operator~\eqref{eq:Koopman} can be numerically instable, due to a bad conditioning of the kernel matrix~$K_\mathcal{X}$. 
To alleviate this, one may include regularization by replacing ${K}_\mathcal{X}^{-1}$ with $({K}_\mathcal{X}+\lambda I)^{-1}$ with regularization parameter $\lambda > 0$. 
Error bounds in the spirit of Theorem~\ref{thm:koehne} for this regularized surrogate were proven in \cite[Thm. 2.4]{bold2024kernel}.
\\

\noindent \textbf{KEDMD for control systems}.
We briefly recall the kernel-EDMD scheme proposed in~\cite{bold2024kernel} for control-affine systems
\begin{align}\label{eq:dynamics:control}
    x^+ = F(x, u) = g_0(x) + G(x) u
\end{align}
with locally Lipschitz-continuous maps $g_0:\Omega \rightarrow \R^n$ and $G:\Omega \rightarrow \R^{n \times m}$ with Lipschitz constants $L_{g_0}, L_G > 0$ and control input $u$ restricted to a compact set $\mathbb{U} \subset \R^m$ containing the origin in its interior. 
An alternative approach that shares the advantage of flexible state-control sampling was proposed in~\cite{bevanda2024nonparametric} using an additional kernel function to express the dependency on the control input. 
Therein, however, no error bounds were provided. 
In addition to rigorous and uniform error bounds, which are crucial for provable guarantees for data-driven MPC, the approach proposed in~\cite{bold2024kernel} allows to counteract numerical ill-conditioning by decomposing the approximation process in two steps labeled as macro and micro level as explicated in the following.
\\

\begin{assumption}[Data requirements]\label{a:data}
    For $\mathcal{X}$ defined by~\eqref{eq:Xcal} with $x_1 = 0$ and cluster radius $\varepsilon_c > 0$, let for each $i \in [1:d]$ data triplets $(x_{ij}, u_{ij}, x_{ij}^+)$, $j \in [1:d_i]$ with $d_i \geq m + (1 - \delta_{1i})$, be given such that
    \begin{equation}\label{eq:ass:data}
        {x}_{ij} \in \mathcal{B}_{\varepsilon_c}(x_i) \cap \Omega, \qquad x_{ij}^+ = F(x_{ij}, u_{ij}), \qquad \operatorname{rank}([u_{i1} \mid \ldots \mid u_{id_i}]) = m.
    \end{equation}
\end{assumption}

Under Assumption \ref{a:data}, we have overall
$D:=\sum\nolimits_{i=1}^{d} d_i$ data triplets. 
We call $x_i$ a virtual observation (or cluster) point to emphasize that no samples at $x_i$ are required. We define the matrices
\[
    U_i := \begin{bmatrix} 
        1 & \dots & 1 \\
        u_{i1} & \dots & u_{id_i}
    \end{bmatrix}, \qquad i \in [1:d].
\]
In the first step, the data points $(x_{ij}, u_{ij}, x^+_{ij})$, $i \in [1:d_i]$, are used to compute an approximation $H_i = [\tilde{g}_0(x_i) \mid \tilde{G}(x_i)] \in \mathbb{R}^{n \times (m+1)}$ of the matrix $[g_0(x_i) \mid G(x_i)]$ for each $x_i \in \mathcal{X}$ by solving the linear regression problem 
\begin{align*}
    \argmin_{H_i} \big\| [x^+_{i1} \mid \ldots \mid x^+_{id_i}] - H_i U_i \big\|_F.
\end{align*}
The solution may be expressed by the pseudoinverse~$U_i^\dagger$, which, for every virtual observation point~$x_i$, is well defined in view of our imposed full row rank of the control samples.

In the second step, the interpolation coefficients are computed analogously to the autonomous case, which leads to the following propagation step of an observable~$\psi$:
\begin{align}\label{eq:kEDMD_model}
    \psi(F(x,u)) \approx \psi^+_\varepsilon(x) := \sum_{i = 1}^d \Big[ (\widehat{K}_0\psi_\mathcal{X})_i \Phi_{x_i}(x) + \sum_{j = 1}^m (\widehat{K}_j \psi_\mathcal{X})_i \Phi_{x_i}(x) u_j \Big],
\end{align}
with $\widehat{K}_j = K_\mathcal{X}^{-1}K_{\tilde{g}_j(\mathcal{X})}K_\mathcal{X}^{-1}$, where $(K_{\tilde g_j(\mathcal{X})})_{k,l} = \k(x_k,\tilde g_j(x_l)) = \k(x_k, (H_l)_{:,j+1})$ for all $x_k,x_l\in \mathcal{X}$ and $j \in [0:m]$.
Then, we directly construct the state-space surrogate of system~\eqref{eq:dynamics:control}, i.e., 
\begin{align}\label{eq:dynamics:kEDMD:control}
    x^+ = \F(x, u) = \g_0(x) + \G(x) u,
\end{align}
using the approximants~\eqref{eq:kEDMD_model} of the Koopman operator with $\psi(x) = x_\ell$, that is,
\begin{align*}
    g_0^\varepsilon(x)_\ell := \sum_{i=1}^d(\widehat K_{0} (x_\ell)_\mathcal{X})_i \Phi_{x_i}(x),\quad (G^\varepsilon(x))_{\ell,j} = \sum_{i=1}^d (\widehat{K}_j (x_\ell)_\mathcal{X})_i \Phi_{x_i}(x).
\end{align*}
for $\ell \in [1:n]$ and $(\ell,j) \in [1:n]\times [1:m]$, respectively.
The $\varepsilon$~notation corresponds to the following bound on the approximation accuracy, which is an adaptation of Theorem~4.3 in~\cite{bold2024kernel}.
\begin{theorem}[Approximation error]\label{thm:control_main}
    Let the smoothness degree of the Wendland kernel function be $k \geq 1$. 
    Furthermore, let Assumption~\ref{a:data} on the data set hold. 
    Then, there exist constants $C_1, C_2, h_0 > 0$ such that, if the fill distance~$h_\mathcal{X}$ satisfies $h_\mathcal{X} \leq h_0$, the error bound 
    \begin{align}\label{eq:prop error:contr}
        \|F(x,u) - \F(x,u)\|_\infty & \le C_1 \varepsilon_{h_\mathcal{X}} \operatorname{dist}(x, \mathcal{X}) + C_2 c \|K_\mathcal{X}^{-1}\|\varepsilon_c =: \varepsilon
    \end{align}
    holds for all $(x,u) \in \Omega \times \mathbb{U}$. Thereby, the error bound depends on the cluster size $\varepsilon_c$, $\varepsilon_{h_\mathcal{X}} = h_\mathcal{X}^{k-1/2}$ 
    and $c$ which is defined by

    \begin{align*}
        c := \max_{i \in [1:d]} \| U_i^\dagger\| \cdot \Phi_{n,k}^{1/2}(0)\big( \max_{v \in \mathbb{R}^d: \| v \|_\infty \leq 1} v^\top {K}_\mathcal{X}^{-1}v \big)^{1/2}.
    \end{align*}
    
\end{theorem}
\begin{proof}
    In a first step, completely analogous to \cite[Proof of Theorem 4.3]{bold2024kernel}, we obtain 
    \begin{equation*}\label{e:step1a}
        |H_{pq}(x_i) - (H_i)_{pq}| \leq \sqrt{2 d_i}(L_{g_0} + L_G \max \{ \|u\|_\infty : u \in \mathbb{U} \}) \| U_i^\dagger \| \cdot \varepsilon_c 
    \end{equation*}
    for $i \in [1:d]$, where 
    $H:\R^n \to \R^{n \times m}$ is defined by $H(x)=[g_0(x)\,|\,G(x)]$.
    Due to uniform continuity of the kernel~$\k$ on
    $\overline{\Omega}\times \overline{\Omega}$, we get the estimate
    \begin{align}\label{eq:est:K}
        \|K_{\tilde{g}_j(\mathcal{X})} - K_{{g}_j(\mathcal{X})}\| &\leq {c_\k} [ L_{g_0} + L_G \max\{\|u\|_\infty : u\in\mathbb{U}\} ] \cdot \max \{ \sqrt{2 d_i} \| U_i^\dagger\| \mid i \in [1:d] \} \varepsilon_c 
    \end{align}
    for a constant ${c_\k}\geq 0$ depending on the continuity modulus of~$\k$. 
    Consequently, this implies by inserting~\eqref{eq:est:K}
    \begin{align*}
        \| \widehat{K}_j - K_\mathcal{X}^{-1}K_{g_j(\mathcal{X})}K_\mathcal{X}^{-1}\| &= \| K_\mathcal{X}^{-1}(K_{\tilde{g}_j(\mathcal{X})} - K_{g_j(\mathcal{X})})K_\mathcal{X}^{-1}\|\\
        &\leq \|K_\mathcal{X}^{-1}\|^2 \|K_{\tilde{g}_j(\mathcal{X})} - K_{{g}_j(\mathcal{X})}\|\\
        &\leq c_\k \tilde{c} \max_{i \in [1:d]} \| U_i^\dagger\| \|K_\mathcal{X}^{-1}\|^2 \varepsilon_c 
    \end{align*}
    with $\tilde{c} = [ L_{g_0} + L_G \max\{\|u\|_\infty : u\in\mathbb{U}\} ] \cdot \max \{ \sqrt{2 d_i} \mid i \in [1:d] \}$.
    
    With \cite[Theorem 11.17]{Wend04} with multiindex~$|\alpha| = 1$, we get the bound
    \begin{align*}
        |(K_\mathcal{X}^{-1} K_{g_j(\mathcal{X})} K_\mathcal{X}^{-1}\varphi_\mathcal{X})^\top\textbf{k}_\mathcal{X}(x) - \varphi \circ g_j(x) | \leq \hat{c}h_\mathcal{X}^{k-1/2}\|\varphi\|_{\mathbb{H}} \ {\operatorname{dist}(x,\mathcal{X})}
    \end{align*}
    with $\hat{c}\geq 0$ for all $\varphi \in \mathbb{H}$
    and $\textbf{k}_\mathcal{X} = (\Phi_{x_1}, \dots, \Phi_{x_d})^\top$. Thus, by the triangle inequality, we obtain a bound on the full error:
    \begin{align*}
        |(\widehat{K}_j\varphi_\mathcal{X})^\top\textbf{k}_\mathcal{X}(x)  - \varphi \circ g_j(x) |
        & \leq \Bigg( \hat{c}h_\mathcal{X}^{k-1/2}\operatorname{dist}(x,\mathcal{X}) \\
        &\hphantom{\quad} 
        + c_\k \tilde c \max_{i \in [1:d]} \| U_i^\dagger\| \ \Phi_{n,k}^{1/2}(0) 
        \max_{\| v \|_\infty \leq 1} \|v\|_{K_\mathcal{X}^{-1}}  \|K_\mathcal{X}^{-1}\|  \varepsilon_c 
        \Bigg) \|\varphi\|_{\mathbb{H}},
    \end{align*}
    where the norm~$\|v\|_{K_\mathcal{X}^{-1}}^2$ is given by $v^\top K_\mathcal{X}^{-1}v$. 
    Last, choosing the coordinate functions as observables, we may from now on assume that $\varphi$ is linear. Thus,
    \begin{align*}
        \varphi(x^+) = \varphi(g_0(x) + G(x)u) = \varphi(g_0(x)) + \sum_{j=1}^m\varphi(g_j(x))u_i
    \end{align*}
    and hence, we get \eqref{eq:prop error:contr} with $C_1 = \hat{c}\max_{i\in [1:n]}\|x_i\|_\mathbb{H}$ and $C_2 = c_\k\tilde{c}\max_{i\in [1:n]}\|x_i\|_\mathbb{H}$,
    where $x_i$, $i\in [1:n]$ denotes the $i$-th coordinate map.
\end{proof}

As can be seen in Theorem~\ref{thm:control_main},the approximation quality depends on the fill distance~$h_\mathcal{X}$ of the grid~$\mathcal{X}$ of cluster points and on the cluster radius $\varepsilon_c > 0$. The choice of the virtual observation points not only influences the first term of the error bound via the fill distance, but also $\|K_\mathcal{X}^{-1}\|$ and $c$ in the second term. Depending on the choice of $\mathcal{X}$, the cluster radius~$\varepsilon_c$ has to be chosen sufficiently small to compensate for $C_2c\|K_\mathcal{X}^{-1}\|$.

Given an exact approximation of $g_0$ and $G$ at the grid points in $\mathcal{X}$, i.e., $\varepsilon_c = 0$, the error bound~\eqref{eq:prop error:contr} is proportional to the distance to the grid points. If, in addition, e.g., $x_1=0$ holds, then the approximation error~\eqref{eq:prop error:contr} is bounded by
\begin{align*}
    C \varepsilon_{h_\mathcal{X}} \operatorname{dist}(x, \mathcal{X}) \leq C\varepsilon_{h_\mathcal{X}} \|x\|.
\end{align*}
\begin{remark}
    If a control-affine continuous-time system 
    \begin{align*}
        \dot{x}(t) := \bar{f}(x(t),u(t)) := \bar{g}_0(x(t)) + \bar{G}(x(t)) u(t)
    \end{align*} 
    is given, a corresponding discrete-time sampled-data system with zero-order hold and time step $\Delta t > 0$
    \begin{equation}\nonumber
        x^+ = x + \int_0^{\Delta t} \bar{g}_0(x(s;x,u))\,\mathrm{d}s + \int_0^{\Delta t} \bar{G}(x(s;x,u))\bar{u} \,\mathrm{d}s
    \end{equation}
    can be derived with $u(s) \equiv \bar{u} \in \mathbb{U}$, assuming existence and uniqueness of the solution $x(\cdot;x^0,u)$ on $[0,\Delta t]$. The control-affine property is approximately preserved for the Koopman operator resulting in an additional error of magnitude~$\mathcal{O}(\Delta t^2)$, see \cite[Rem.~4.2]{bold2024kernel}. 
\end{remark}

\section{MPC via kernel EDMD}

\noindent In this section, we propose kEDMD-MPC, which leverages the data-driven surrogate~\eqref{eq:dynamics:kEDMD:control} 
in the optimization step.

First, we define admissibility of sequences of control values of length~$N$ taking into account the set~$\mathbb{X} \subseteq \Omega$ representing the state constraints. 
In the following, the set $\mathbb{X}$ with $0 \in \operatorname{int}(\mathbb{X})$ is assumed to be convex and compact.

\begin{definition}\label{def:admissibility}
    A control sequence $u = (u(k))_{k=0}^{N-1} \subset \mathbb{U}$ of length~$N$ is said to be \textit{admissible} for state $\hat{x} \in \mathbb{X}$, if $x_u(k) \in \mathbb{X}$ holds for all $k \in [1:N]$, where 
    \[
        x_u(k) = F(x_u(k-1),u(k-1))
    \]
    using the dynamics~\eqref{eq:dynamics:control} and $x_u(0) = \hat{x}$. 
    For $\hat{x} \in \mathbb{X}$, the set of admissible control sequences is denoted by~$\mathcal{U}_N(\hat{x})$. 
    If, for $u = (u(k))_{k=0}^\infty$ and $(u(k))^{N -1}_{k = 0} \in \mathcal{U}_N(\hat{x})$ holds for all $N \in \N$, 
    we write $u \in \mathcal{U}_\infty(\hat{x})$.
\end{definition}
If the surrogate~$\F$ is considered, we change Definition~\ref{def:admissibility} by tighening the state constraint, i.e., imposing
\begin{equation}\label{def:admissiblity:varepsilon}
    x_u^\varepsilon(k) = \F(x_u^\varepsilon(k-1),u(k-1)) \in \mathbb{X} \ominus \mathcal{B}_{k \varepsilon}(0),
\end{equation}
where $\ominus$ denotes the usual Pontryagin set difference. 
We call these control sequences admissible for $\hat{x} \in \mathbb{X}$ w.r.t.\ the surrogate~\eqref{eq:dynamics:kEDMD:control} denoted by $u \in \mathcal{U}_N^\varepsilon(\hat{x})$.

In the following, we will use the quadratic stage costs
\begin{align}\label{eq:stagecosts}
    \ell(x, u) = \|x\|_Q^2 + \|u\|_R^2 = x^\top Q x + u^\top R u,
\end{align} 
symmetric positive definite $Q \in \R^{M \times M}$, $R\in \R^{m \times m}$.
The proposed kEDMD-MPC method is summarized in Algorithm~\ref{alg:kEDMD-MPC}.
\begin{algorithm}[htb]
    \caption{kEDMD-MPC}\label{alg:kEDMD-MPC}
    \raggedright
    \smallskip\hrule
    \smallskip
    {\it Input:} Horizon $N \in \N$, 
    surrogate~$\F$, constraints~$\mathbb{U}$, $\mathbb{X}$. 
    \smallskip\hrule
    \medskip
    \textit{Initialisation}: Set $k = 0$.\\[2mm]
    \noindent\textit{(1)} Measure current state $x_{\mu_N^\varepsilon}(k)$ and set $\hat{x} = x_{\mu_N^\varepsilon}(k)$.\\[1mm]
    \noindent\textit{(2)} Solve the optimal control problem~
    \begin{align}\label{eq:OCP}\tag{OCP}
    \begin{split}
        \min_{\bar{u} \in \mathcal{U}_N^\varepsilon(\hat{x})} \quad & \sum\nolimits_{i = 0}^{N-1} \ell(x_{\bar{u}}^\varepsilon(i), \bar{u}(i)) \\
        \text{s.t.} \quad & x_{\bar{u}}^\varepsilon(0) = \hat{x} \text{ and for $i \in [1:N-1]$} \\
        & x_{\bar{u}}^\varepsilon(i+1) = \F(x_{\bar{u}}^\varepsilon(i), \bar{u}(i))
    \end{split}
    \end{align}
    \hspace*{5mm} to obtain the
    optimal control sequence~$(\bar{u}^\star(i))_{i = 0}^{N - 1}$. \\[1mm]
    \noindent\textit{(3)} Apply the MPC feedback law~$\mu_N^\varepsilon(x(k)) = \bar{u}^\star({0})$ at the\\
    \hspace*{5mm} plant, shift $k = k + 1$, and go to Step~(1).
    \smallskip\hrule
\end{algorithm}

A common method to ensure closed-loop stability is the use of suitable terminal ingredients, see \cite{magni2001stabilizing}. However, the design of terminal ingredients can be a demanding task in the nonlinear framework. Another option to achieve stability without terminal ingredients relies on cost controllability, a stabilizability condition, in combination with a sufficiently long prediction horizon, see, e.g., \cite{boccia2014stability}.

The (optimal) value function~$\V_N: \mathbb{X} \rightarrow \mathbb{R}_{\geq 0} \cup \{ \infty \}$ for the MPC algorithm associated to the optimal control problem~\eqref{eq:OCP} is defined by 
\begin{align*}
    \V_N(\hat{x}) := \inf_{u \in \U(\hat{x})} \J_N(\hat{x},u)
\end{align*} 
with $\J_N(\hat{x},u) = \sum_{k=0}^{N - 1} \ell(x^\varepsilon_u(k),u(k))$, where $x^\varepsilon_u(0) = \hat{x}$.

\section{Stability of kEDMD-MPC}

\noindent In this section, we prove practical asymptotic stability of the origin w.r.t.\ the closed-loop dynamics of kEDMD-MPC as defined in Algorithm~\ref{alg:kEDMD-MPC}.

\begin{definition}
\label{def:PAS}
    Let $\mu^\varepsilon_N$ be the feedback law defined in Algorithm~\ref{alg:kEDMD-MPC} and EDMD admissibility of a control sequence $(u(k))_{k = 0}^{N - 1}$ at $\hat{x}$ be defined by~\eqref{def:admissiblity:varepsilon}. 
    The origin is said to be $\varepsilon_0$-\textit{practically asymptotically stable} (PAS) on~$S \subseteq \mathbb{X}$ if there exist $\beta \in \mathcal{KL}$ such that: $\forall\,r>0$ $\exists\,\varepsilon_0 > 0$ such that, for each fill distance~$h_\mathcal{X}$ and cluster radius~$\varepsilon_c$ such that \eqref{eq:prop error:contr} holds with $\varepsilon \in (0,\varepsilon_0]$, \eqref{eq:OCP} is feasible for all $\hat{x} \in S$ and the MPC closed-loop solution $x_{\mu^\varepsilon_N}(\cdot)$ generated by $x_{\mu^\varepsilon_N}(0) = \hat{x}$ and 
    \begin{align}\label{eq:ex_cl}
        x_{\mu^\varepsilon_N} (k+1) = F(x_{\mu^\varepsilon_N}(k),\mu^\varepsilon_N(x_{\mu^\varepsilon_N}(k)))
    \end{align}
     satisfies $x_{\mu_N}^\varepsilon(k) \in S$ and 
        $\| x^{\varepsilon}_{\mu_N}(k) \| \leq \max \{\beta(\| \hat{x} \|,k),r\}$ 
    for all $k \in \mathbb{N}_0$.
\end{definition}

In the following, we prove PAS of the origin via \cite[Theorem~11.10]{grune2017nonlinear}, which serves as a flexible tool to verify PAS subject to bounded model perturbations. 
With the same strategy, PAS for EDMD-based MPC was proven in \cite{bold2024data} under a restrictive invariance condition on the dictionary, which is not needed for kEDMD-MPC as shown in the following.

The application of \cite[Theorem~11.10]{grune2017nonlinear} requires the following 
key assumptions: 
\begin{itemize}
    \item[1)] Continuity of $\F$ uniform in $u\in \mathbb{U}$ on~$\Omega$, i.e., for some $\varepsilon_0 > 0$ there is $\omega \in \mathcal{K}$ such that 
    \begin{align*}
        \|\F(x,u) - \F(y,u)\| \leq \omega(\|x-y\|) 
    \end{align*}
    for all $x, y \in \Omega$, $u\in \mathbb{U}$ and $\varepsilon \in (0,\varepsilon_0]$.
    \item[2)] Relaxed Lyapunov 
    inequality: 
        For a forward invariant set~$S \subset \mathbb{X}$ w.r.t.\ $\F$, there exist $\varepsilon_0 > 0$ and $\alpha \in (0, 1]$ such that
        \begin{align*}
             \V_N(\F(x, \mu^\varepsilon_N(x))) \leq \V_N(x) - \alpha \ell(x, \mu^\varepsilon_N(x)).
        \end{align*}
        for all $\varepsilon \in (0, \varepsilon_0]$ and $x \in S$, 
    \item[3)] Uniform continuity of $\V_N$ and existence of $\alpha_1, \alpha_2, \alpha_3 \in \mathcal{K}_\infty$ such that $\alpha_1(\|x\|) \le {V}_N^\varepsilon(x) \le \alpha_2(\|x\|)$ and $\ell(x, u) \ge \alpha_3(\|x\|)$ for all $x \in S$ and $u \in \mathbb{U}$.
\end{itemize}

In Lemma~\ref{lem:uniformcont}, we show uniform continuity of the surrogate~$\F$ for cluster radius~$\varepsilon_c = 0$, i.e., Property~1. In Proposition~\ref{prop:cost:controllability}, we leverage the proportional bound of Theorem~\ref{thm:control_main} to show that cost controllability is preserved in the kEDMD surrogate. 
This allows to infer Property~2) for a sufficiently long prediction horizon in Theorem~\ref{thm:PAS}. 
Leveraging these findings enables us to ensure the remaining Property~3) and, thus, PAS analogously to \cite[Theorem~10]{bold2024data}.

\begin{lemma}[Uniform continuity of surrogate~$\F$]\label{lem:uniformcont}
Let {$k \geq 1$} 
and $F \in C_b^{\ceil{\sigma_{n,k}}}(\Omega;\R^{n})$ with $\sigma_{n,k} := \frac{n+1}{2} + k$ and $\varepsilon_c=0$. Then, there exist a constant $h_0 > 0$ such that, if the fill distance~$h_\mathcal{X}$ satisfies $h_\mathcal{X} \leq h_0$, the surrogate~$\F$ defined by~\eqref{eq:dynamics:kEDMD:control} is uniformly continuous in~$x$ w.r.t.\ $u \in \mathbb{U}$ with modulus of continuity~$\omega(r) = (C_{\nabla \F} h_\mathcal{X}^{k-1/2} + \widetilde{C}) r$ for constant~$C_{\nabla \F}>0$ and $\widetilde{C} = \sup\limits_{z \in \Omega} \|\nabla_x F(z, u)\|$, i.e.,
    \begin{align*}
        \| \F(x, u) - \F(y, u)\| \le \omega(\|x - y\|)
    \end{align*}
    holds for all $x, y \in \Omega$ and $u \in \mathbb{U}$. 
\end{lemma}
\textbf{Proof.}        
Let $x,y \in \Omega$ and $u \in \mathbb{U}$ be given. Then, for all $i \in [1:n]$, it holds by using Taylor's theorem 
    \begin{align*}
    \F(x, u) - \F(y, u) &= [\F(y, u) + (x - y)^\top\nabla_x \F(\xi,u)] - \F(y, u) \\
    &= (x - y)^\top \nabla_x \F(\xi,u)
    \end{align*}
for a $\xi \in \{\tilde{\xi} \in \R^n \mid \tilde{\xi} = (t - 1)x + t y, \ t \in [0, 1]\}$.
Consequently, we may estimate
\begin{align*}
    \|\F(x, u) - \F(y, u)\| &\leq \|x - y\| \|\nabla_x \F(\xi, u) \pm \nabla_x F(\xi, u)\| \\
    & \leq \|x - y\|\left(\|\nabla_x \F(\xi, u) - \nabla_x F(\xi, u)\| + \|\nabla_x F(\xi, u)\|\right).
\end{align*}
To compute a bound for $\|\nabla_x \F(\xi, u) - \nabla_x F(\xi, u)\|$, we first leverage the control-affine structure and, then, 
apply 
{\cite[Theorem 11.17]{Wend04}
with multiindex~$|\alpha| = 1$} invoking $\varepsilon_c = 0$ to obtain
\begin{align*}
    \|\nabla_x \F(\xi, u) - \nabla_x F(\xi, u)\| & \le \|\nabla_x \g_0(\xi) - \nabla_x g_0(\xi)\| + \sum\nolimits_{i = 1}^m \|\nabla_x \g_i(\xi) - \nabla_x g_i(\xi)\| \|u\|_\infty \\
    & \le ( C_{\nabla_x g_0} + \|u\|_\infty \sum_{i = 1}^m C_{\nabla_x g_i} ) \cdot h_\mathcal{X}^{k-1/2} =: C_{\nabla \F} h_\mathcal{X}^{k-1/2}
\end{align*}
for all $\xi \in \Omega$, where $g_i$ stands for the $i$-th column of $G$. 
Since the function~$F$ is continuously differentiable w.r.t.\ its first argument on the compact set~$\Omega$, we have $\| \nabla_x F(\xi, u)\| \leq \widetilde{C} \in (0,\infty)$. 
Overall, we have the following inequality
\begin{align*}
     \|\F(x, u) - \F(y, u)\| &\leq \|x - y\| \left(C_{\nabla_x \F} h_0^{k-1/2} + \widetilde{C} \right) = \omega (\|x - y\|). 
\end{align*}
\qed

The continuity in Lemma~\ref{lem:uniformcont} is uniform in the approximation bound~$\varepsilon>0$. This is in contrast to finite-element dictionaries, see~\cite{schaller2023towards}, where the derivative of the ansatz functions increases for decreasing mesh size.

In Proposition~\ref{prop:cost:controllability}, we show that cost controllability is preserved for the data-driven surrogate for $\varepsilon_c = 0$, i.e., cluster radius zero. This allows to deduce the relaxed dynamic programming inequality. 
\begin{proposition}[Cost controllability]\label{prop:cost:controllability}
    For $C, \varepsilon_0 > 0$, let the data requirements~\eqref{eq:ass:data} and the $\varepsilon$-bound~\eqref{eq:prop error:contr} depending on fill distance~$h_{\mathcal{X}}$ and $\varepsilon_c = 0$ hold, i.e., we have the proportional error bound
    \begin{align}\label{eq:error:bound:proportional}
        \|F(x,u) - \F(x,u)\|_\infty & \leq C \varepsilon_{h_\mathcal{X}} \operatorname{dist}(x, \mathcal{X}).
    \end{align}    
    for all $\varepsilon \in (0,\varepsilon_0]$.
    Suppose that the system~$F$ governed by~\eqref{eq:dynamics:control} with stage cost~\eqref{eq:stagecosts} is \textit{cost controllable} on the set of $\varepsilon$-admissible states $\mathbb{X}^\varepsilon := \{ x \mid \mathcal{U}_\infty^\varepsilon(x) \neq \emptyset\}$, i.e., there exists a monotonically increasing and bounded sequence $(B_k)_{k \in \N}$ such that, for every $\hat{x} \in \mathbb{X}^\varepsilon$, $\exists\,u \in \mathcal{U}_{\infty}^\varepsilon(\hat{x})$ satisfying the growth bound\footnote{The first inequality directly follows from the definition of the value function and is only stated to link~\eqref{eq:growthbound} to cost controllability as in~\cite{GrunPann10}, see~\cite{Wort12} and the references therein.}
    \begin{align} \label{eq:growthbound}
        V_k(\hat{x}) \le J_{\bar{N}}(\hat{x}, u) \le B_{\bar{N}} \ell^*(\hat{x}) \qquad\forall\,\bar{N} \in \N
    \end{align}
    with $\ell^*(\hat{x}) := \inf_{u\in \mathbb{U}} \ell(\hat{x}, u)$.
    Then, cost controllability holds for the $\varepsilon$-surrogate~$\F$ defined by~\eqref{eq:dynamics:kEDMD:control} on $\mathbb{X}^\varepsilon$, i.e., existence of a monotonically increasing and bounded sequence $(\B_k)_{k \in \N}$ such that Inequality~\eqref{eq:growthbound} holds for $\V_{\bar{N}}$, $\J_{\bar{N}}$, $\mathcal{U}_{\bar{N}}^\varepsilon$ instead of $V_{\bar{N}}$, $J_{\bar{N}}$, $\mathcal{U}_{\bar{N}}$, respectively. 
    Moreover, we have $\B_{\bar{N}} \rightarrow B_{\bar{N}}$ for $h_\mathcal{X} \rightarrow 0$ for all $k \in \N$. 
\end{proposition}
\noindent\textbf{Proof.} 
    Set $\bar{\lambda} = \max\{|\lambda| : \lambda \text{ eigenvalue of } R \text{ or } Q\}$ and $0 < \underline{\lambda} = \min\{|\lambda| : \lambda \text{ eigenvalue of } R \text{ or } Q\}$. 
    Let $N \in \mathbb{N}$, $\hat{x} \in \mathbb{X}^\varepsilon$, and $u \in \mathcal{U}_{N}^\varepsilon(\hat{x})$ satisfying the growth bound~\eqref{eq:growthbound} be given.
    Then, we have
    \begin{eqnarray}
        & & \ell(x_u^\varepsilon(k), u(k)) - \ell(x_u(k), u(k)) \label{eq:cost:controllability:proof0} \\
        & = & \| x_u^\varepsilon(k) \mp x_u(k) \|^2_Q + \|u(k)\|^2_R - \ell(x_u(k), u(k)) \nonumber \\
        & \leq & \bar{\lambda} \| x_u^\varepsilon(k) - x_u(k) \|^2 + 2 \bar{\lambda} \| x_u^\varepsilon(k) - x_u(k) \| \|x_u(k)\|. \nonumber
    \end{eqnarray}
    
    Then, using $\operatorname{dist}(x,\mathcal{X}) \leq \| x \|$ in view of $x_1 = 0 \in \mathcal{X}$,  the difference $e_{k+1} := \| x_u^\varepsilon(k + 1) - x_u(k + 1) \|$ is estimated using the triangular inequality by
    \begin{align}
        & e_{k + 1} \hspace*{-0.75mm}=\hspace*{-0.75mm} \| \F(x_u^\varepsilon(k), u(k)) \hspace*{-0.5mm}\pm\hspace*{-0.5mm} F(x_u^\varepsilon(k), u(k)) \hspace*{-0.5mm}-\hspace*{-0.5mm} {F}(x_u(k), u(k))\| \nonumber \\
        & \phantom{e_{k + 1}} \hspace*{-0.75mm} \leq C \varepsilon_{h_\mathcal{X}} \| x_u^\varepsilon(k) \pm x_u(k)\| + L_F\| x_u^\varepsilon(k) - x_u(k) \| \nonumber \\
        & \phantom{e_{k + 1}} \hspace*{-0.75mm} \leq C \varepsilon_{h_\mathcal{X}} \|x_u(k)\| + (L_F + C \varepsilon_{h_{\mathcal{X}}}) e_k \label{eq:cost:controllability:proof1} \\
        & \phantom{e_{k + 1}} \hspace*{-0.75mm} \leq C \varepsilon_{h_\mathcal{X}} \sum\nolimits_{i=0}^k (L_F+C \varepsilon_{h_{\mathcal{X}}})^{k-i} \|x_u(i)\|. \label{eq:cost:controllability:proof2}
    \end{align}
    Hence, using inequality~\eqref{eq:cost:controllability:proof1}, we get 
    \begin{eqnarray*}
        e_k^2 & \leq & 2 C^2 \varepsilon_{h_{\mathcal{X}}}^2 \| x_u(k - 1)\|^2 + 2 (L_F + C \varepsilon_{h_\mathcal{X}})^2 e^2_{k - 1} \\
        & \leq & \frac{ C^2 \varepsilon_{h_{\mathcal{X}}}^2}{\underline{\lambda}} \sum\nolimits_{i = 0}^{k - 1} [2 (L_F + C \varepsilon_{h_\mathcal{X}})^2]^{k - 1 - i} \ell(x_u(i), u(i)).
    \end{eqnarray*}
    Analogously, leveraging inequality~\eqref{eq:cost:controllability:proof2} yields
    \begin{align*}
        e_k \|x_u(k)\| & \leq C \varepsilon_{h_\mathcal{X}} \sum\nolimits_{i=0}^k (L_F+C \varepsilon_{h_{\mathcal{X}}})^{k-i} \|x_u(i)\| \|x_u(k)\|.
    \end{align*}
    Using the last two inequalities and 
    \begin{align*}
        2 \|x_u(i)\| \|x_u(k)\| & \leq \|x_u(i)\|^2 + \|x_u(k)\|^2
    \end{align*}
    in inequality~\eqref{eq:cost:controllability:proof0} yields, for $\ell(x_u^\varepsilon(k), u(k))$ summed up over $k \in [0 : N - 1]$, the estimate\footnote{Note that the first summands in $\J_N(\hat{x},u)$ and $J_N(\hat{x},u)$ coincide.}
    \begin{align*}
         \J_N(\hat{x}, u) &\leq J_N(\hat{x}, u) + \bar{\lambda} \Big( \sum\nolimits_{k = 0}^{N - 1} e_k^2 + 2 e_k\|x_u(k)\| \Big) \\
        &\leq  \left( B_N + c_1^\varepsilon C \varepsilon_{h_{\mathcal{X}}} \bar{\lambda}/\underline{\lambda} + c_2^\varepsilon C^2 \varepsilon_{h_{\mathcal{X}}}^2 \bar{\lambda}/\underline{\lambda} \right)\ell^*(\hat{x}) \\&=: B^\varepsilon_N\ell^*(\hat{x})
    \end{align*}
    with $c_1^\varepsilon = \sum_{k = 0}^{N - 1} B_k ( \sum_{i=0}^k d^{i}) (\sum_{i=0}^{N-1-k} d^i)$, $B_0 = 1$, and $c_2^\varepsilon = \sum_{k = 0}^{N-2} B_k \sum_{i=0}^{N-2-k} d^{i}$, where we have invoked the imposed cost controllability multiple times and used the abbreviation $d = 2(L_F+C \varepsilon_{h_{\mathcal{X}}})^2$. 
    Moreover, $B_N^\varepsilon \rightarrow B_N$ for $h_{\mathcal{X}} \rightarrow 0$ as claimed.
\qed

The following theorem combines the results of Theorem~\ref{thm:control_main}, Lemma~\ref{lem:uniformcont}, and Proposition~\ref{prop:cost:controllability} to deduce PAS of kEDMD-MPC. 
In essence, it states that if the nominal MPC controller asymptotically stabilizes the origin, the kEDMD-MPC controller ensures convergence to a neighborhood of the origin, whose size depends on the approximation accuracy~$\varepsilon$ (and thus on the fill distance~$h_\mathcal{X}$). The proof resembles the proof of Theorem~10 in~\cite{bold2024data}.
\begin{theorem}\label{thm:PAS}
     Let the assumptions of Proposition~\ref{prop:cost:controllability} holds and set $S$ to a level set of~$V_N^\varepsilon$ contained in $\mathbb{X}^\varepsilon$. 
     Let the prediction horizon $N$ be chosen such that $\alpha \in (0, 1)$ with
     \begin{align*}
         \alpha = \alpha_N := 1 - \frac{(B_2 - 1)(B_N - 1)\prod\nolimits_{i = 3}^N(B_i - 1)}{\prod\nolimits_{i = 2}^N B_i - (B_2 - 1)\prod\nolimits_{i = 3}^N(B_i - 1)}.
     \end{align*}
     Then the EDMD-MPC controller of Algorithm~\ref{alg:kEDMD-MPC} ensures $\varepsilon$-PAS of the origin on the set~$S$.
\end{theorem}

\section{Numerical simulations}\label{sec:simulations}

\noindent In this section, we illustrate the practical asymptotic stability of kEDMD-MPC in numerical simulations. To this end, we consider the discrete-time nonlinear control-affine van-der-Pol oscillator
\begin{align}\label{eq:van der Pol}
    x^+ = x + \Delta t \binom{x_2}{\nu(1 - x_1)^2x_2 - x_1 + u} 
\end{align}
with parameters $\Delta t = 0.05$, $\nu = 0.1$ and control constraints $\mathbb{U} = [-2, 2]$ on the domain $\Omega = [-2, 2]^2$. 
To build the corresponding kEDMD surrogate~\eqref{eq:dynamics:kEDMD:control}, we used Wendland kernels with smoothness degree~$k = 1$, see Remark~\ref{rem:wendland}. 
The grid of cluster points~$\mathcal{X}_d$ consists of the Cartesian product of $\sqrt{d}$ one-dimensional Chebyshev nodes
\[
    \Big\{ \hspace*{-1.mm} \left(2 \cos \left(\tfrac{\pi(2i + 1)}{2 \cdot \sqrt{d}}\right)\hspace*{-1.mm}, 2 \cos\left(\tfrac{\pi(2j + 1)}{2 \cdot \sqrt{d}}\right)\right)^{\hspace*{-1mm}\top}\hspace*{-2.5mm}, i, j \in [1:\sqrt{d}]\Big\},
\] 
where $d\in \{441,1681\}$, i.e., $\sqrt{d} \in \{21,41\}$. 
This choice corresponds to a more evenly distributed error throughout the domain compared to a uniform grid, see~\cite{kohne2024infty}. 
For each data point~$x_i \in \mathcal{X}$, $d_i = 25$
random control values $u_{ij} \in \mathbb{U}$ are chosen, yielding the data points~$(x_{ij}, u_{ij}, x_{ij}^+)$. 
Herein, $x_{ij}$ and $u_{ij}$ are drawn according to our data requirements with $\varepsilon_c = \nicefrac{\sqrt{2}}{d}$ to specify the neighborhood of~$x_i$, see Section~\ref{sec:kEDMD}.\\
The MPC cost function is chosen as in~\eqref{eq:stagecosts} with matrices $Q = I_2$ and $R = 10^{-4}$. 
The following closed-loop simulations emanate from the initial state $x^0 = (0.5, 0.5)^\top$.
\begin{figure}[htb]
    \centering
    \includegraphics[width=0.5\linewidth]{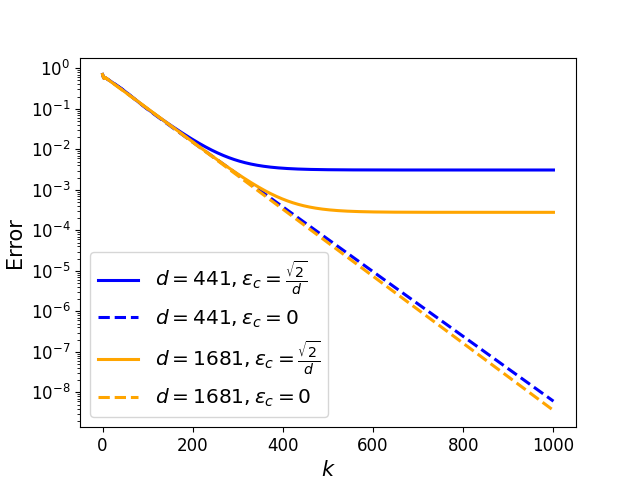}
    \caption{System~\eqref{eq:van der Pol}: Error $\|x(k)\|$ 
    of the kEDMD-MPC closed-loop with horizon $N = 10$ for vaying~$d$ and~$\varepsilon_c$.}
    \label{fig:MPC02}
\end{figure}
For optimization horizon~$N = 10$ and $d \in \{441,1681\}$ cluster points, Figure~\ref{fig:MPC02} shows the course of the closed-loop error for $\varepsilon_c = \nicefrac{\sqrt{2}}{d}$ compared to $\varepsilon_c = 0$, i.e., without an approximation error in the first step of the algorithm described in Section~\ref{sec:kEDMD}. In the beginning, the error decays very similarly until it stagnates for $\varepsilon_c = \nicefrac{\sqrt{2}}{d}$ and asymptotically approaches an offset after about $500$ steps. Moreover, we observe that this offset reaches a lower value for $d = 1681$ than for $d = 441$. In the case $\varepsilon_c = 0$, the error 
does not stagnate and shows
exponential convergence towards the origin for all times.
Figure~\ref{fig:MPC02} validates the practical asymptotic stability result that was proven in Theorem~\ref{thm:PAS} (for $\varepsilon_c = 0$) and shows that the stability rate is exponential. Furthermore, the simulations show that practical asymptotic stability is also obtained in the case $\varepsilon_c = \nicefrac{\sqrt{2}}{d} > 0$. 

\begin{figure}[htb]
    \centering
    \includegraphics[width=0.5\linewidth]{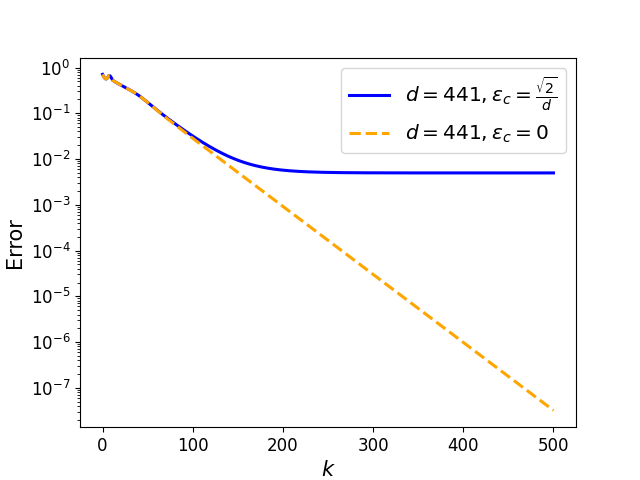}
    \caption{System~\eqref{eq:van der Pol}: Error $\|x(k)\|$ 
    of the kEDMD-MPC closed-loop with horizon $N = 20$ with varying~$\varepsilon_c$.}
    \label{fig:MPC02:N20}
\end{figure}

For the last considerations depicted in Figure~\ref{fig:MPC02:N20}, we increase the optimization horizon to $N = 20$ and fix the number of cluster points $ d=441$. As the value of the offset for $\varepsilon_c = \nicefrac{\sqrt{2}}{d}$ is similar to the case $N=10$, this indicates that the practical convergence stems from the error of the data-driven scheme and not from a too short prediction horizon. Last, we observe that the exponential decay in Figure~\ref{fig:MPC02} is visibly steeper for $N=20$.

\section{Conclusions} 

\noindent We have proposed the novel data-driven predictive control scheme kEDMD-MPC. In particular, we have rigorously shown practical asymptotic stability of data-driven MPC leveraging kernel EDMD, pointwise bounds on the full approximation error, and cost controllability of the original dynamics.

\bibliographystyle{plain}
\bibliography{references}

\end{document}